\documentclass[12pt,leqno,psfig,openbib,hidelinks]{amsart}
\usepackage{amssymb,amsmath,amsthm,graphicx,amsfonts,euscript}
\usepackage{mathtools}
\usepackage{comment}
\usepackage{fancyhdr}
\usepackage{epsfig}
\usepackage{setspace}
\usepackage{bbm}
\usepackage{dsfont}
\usepackage[T1]{fontenc}
\usepackage[utf8]{inputenc}

\usepackage{thmtools}
\usepackage{thm-restate}

\usepackage{xcolor}
\usepackage{hyperref}
\usepackage{cleveref}

\hypersetup{colorlinks=true,linkcolor=blue,citecolor=blue,
	filecolor=magenta, urlcolor=blue}

\newtheorem{prop}{Proposition}[section]
\if@secthm
\newtheorem{thm}{Theorem}[section]
\else
\newtheorem{thm}{Theorem}[section]
\fi

\newtheorem{Lemma}{Lemma}[section]
\definecolor{refkey}{gray}{.75}

\numberwithin{equation}{section}

\usepackage{empheq,etoolbox}
\patchcmd{\subequations}
{\theparentequation\alph{equation}}
{\theparentequation.\alph{equation}}
{}{}

\renewenvironment{proof}{\medskip\noindent{\em Proof.\,\;}}{\hfill$\square$}

\date{\today}
\subjclass[2000]{\ ,\ } \keywords{Random matrix}

\newtheorem*{question*}{Question}
\newtheorem*{claim*}{Claim}
\newtheorem*{remark*}{Remark}
\newtheorem*{application*}{Application}
\DeclareMathOperator{\dist}{dist}
\DeclareMathOperator{\id}{id}
\DeclareMathOperator{\supp}{supp}

\date{\today}
\subjclass[2000]{35Q31, 35B30} \keywords{Generalized Euler equations, data-to-solution map, instability}

\begin{document}
\title[]{Instability of Data-to-Solution Map for the Log-Regularized 2D Euler System}
\author[]{ Xuan-Truong Vu}
\address{Department of Mathematics, Statistics, and Computer Science, University of Illinois at Chicago, Chicago, Illinois 60607, U.S.A.
}
\email[]{tvu25@uic.edu}
\maketitle
\begin{abstract}
	In this paper, we study	the logarithmically regularized $2$D Euler system \eqref{e1}, which is derived by regularizing the Euler equation for the vorticity. We establish local well-posedness of the logarithmically regularized $2$D Euler equations  in the subcritical space $H^s(\mathbb{R}^2)$ with $s>2$ for $\gamma \ge 0$. Furthermore, we show that for $\gamma$ close to $0$, the data-to-solution map is not uniformly continuous in the Sobolev $H^s(\mathbb{R}^2)$ topology for any $s>2$.
\end{abstract}

\rhead[In preparation.]{\textit{In preparation.}}

\section{Introduction}\label{s1}

We consider the initial value problem for the logarithmically regularized $2$D Euler equations, a family of active scalar equations of the form
\begin{equation}\label{e1}
	\begin{cases}
    \begin{split}
      &\partial_t \theta+(u\cdot\nabla)\theta=0,\\
		& u=\nabla^{\perp}\Delta^{-1}T_{\gamma}\,\theta,\\
		&\theta(x,0)=\theta_{T_{\gamma},0} (x),\qquad x\in \mathbb{R}^2,\ t\in \mathbb{R},  
    \end{split}	
	\end{cases}
\end{equation}
where $\nabla^{\perp}=(-\partial_2,\partial_1)$, $\theta = \theta(x,t)$  denotes the scalar vorticity and $u=u(x,t)$ is the velocity field recovered via a Biot–Savart-type law involving the nonlocal operator $T_{\gamma}$. We focus on the regularization given by
\begin{equation}\label{generalTgammaclass}
  T_{\gamma}(|\xi|)=\left(\log(\mathrm{e}+|\xi|^2)\right)^{-\gamma}, \quad \gamma \ge 0,  
\end{equation}
which acts as a logarithmic smoothing operator in Fourier space. When $\gamma=0$, the system reduces to the classical 2D incompressible Euler equations in vorticity formulation: 
\begin{equation}\label{e3}
	\begin{cases}
		\partial_t \theta+(u\cdot\nabla)\theta=0,\\
		u=\nabla ^{\perp}\Delta^{-1}\theta,\\
		
		\theta(x,0)=\theta_{e,0} (x), \quad (x,t)\in \mathbb{R}^2\times \mathbb{R},
	\end{cases}
\end{equation}
 corresponding to the identity operator. 
  The superscript $e$ in $\theta_{e,0}$ above and in all subsequent formulas refers to the solutions of the Euler equations. Solutions of \eqref{e1} corresponding to the multiplier $T_\gamma$ will be denoted with the subscript $T_\gamma$. Logarithmically regularized models of the Euler equations have been extensively studied in recent years as a means of investigating the delicate threshold between well-posedness and ill-posedness in critical and borderline Sobolev spaces. For instance, Chae and Wu \cite{ChaeWu12}, and later Dong and Li \cite{DongLi15}, proved that logarithmic smoothing is sufficient to restore local (and in some cases global) well-posedness in the borderline Sobolev space $H^1(\mathbb{R}^2)\cap\dot{H}^{-1}(\mathbb{R}^2)$ provided the regularization parameter satisfies $\gamma>1/2$. However, a recent result by Kwon \cite{HK} shows that for $0<\gamma<1/2$, the system becomes strongly ill-posed in the same critical space, with norm inflation and failure of continuous dependence even for smooth, compactly supported perturbations.

\medskip

While the borderline regime has thus been completely characterized in terms of well-posedness, relatively less is known about the behavior of solutions in subcritical Sobolev spaces, particularly when $\gamma$ is small but positive. This is the focus of the present paper. In this work, we study the logarithmically regularized 2D Euler equations \eqref{e1} in $\mathbb{R}^2$  and our main results concern the well-posedness and instability of the data-to-solution map in the subcritical Sobolev space $H^s(\mathbb{R}^2)$ for $s>2$. We establish the following two main theorems:

Our results are gathered in the following statements.  
\begin{thm}\label{thm:lwp_Hadmard}
For $\gamma \ge 0$ and initial data $u_0\in H^s$ with $s>2$, there exist a time $t_{*}>0$ such that the system \eqref{e1} admits a unique solution $u \in C(0,t_{*}; H^s)\cap C^1(0,t_{*};H^{s-1})$, which depends continuously on the initial data $u_0$.
 \end{thm}
\begin{restatable}{thm}{nonuniformfornonperiodiccase}\label{thm:nonuniformfornonperiodiccase}
	Let $\gamma \ge 0$. Let $\tilde{\theta}_{T}^{n}$,  $ \tilde{\tilde{\theta}}_{T}^{n}$ where $n=1,2,\dots$ be the solutions of the Cauchy problem \eqref{e1} corresponding to the data  $\tilde{\theta}_{T}^{n}(0)=\tilde{\theta}_e^{n}(0)=\mathrm{curl}\,\tilde{u}_{e,0}^{n}$,  $\tilde{\tilde{\theta}}_{T}^{n}(0)= \tilde{\tilde{\theta}}_e^{n}(0)= \mathrm{curl}\, \tilde{\tilde{u}}_{e,0}^{n}$ belonging to the space $ \dot{H}^{-1}(\mathbb{R}^2)\cap  {H}^{s+1}(\mathbb{R}^2)$ where $s>2$. Let $t_{_*}>0$ be the common lifespan of both solutions. Then, for any $ 0<t\le t_{_*}$ we have 
	\begin{equation}
		\liminf\limits_{n\rightarrow \infty}\|\tilde{u}_{T}^{n}(t)- \tilde{\tilde{u}}_{T}^{n}(t)\|_{H^s}\ge |\sin t|+o(\|T-\id\|)>0\notag
	\end{equation}
	for all $T$ sufficiently close to the identity in the operator norm $\|\cdot\|$.
\end{restatable}

This reveals a novel form of Hadamard-type instability in the subcritical regime: even though the system is locally well-posed, arbitrarily small changes in initial data may lead to substantial deviations in the corresponding solutions. Such phenomena are more commonly associated with critical or borderline spaces (as in the work of Bourgain–Li \cite{BL} and Kwon \cite{HK}), and their emergence in subcritical regimes is both surprising and mathematically significant.  

\medskip

Our results therefore extend the ill-posedness perspective beyond the critical space $H^1$, showing that the logarithmic regularization, while sufficient to control norm inflation in low-regularity settings when $\gamma>1/2$, does not guarantee stable dependence on initial data even in smoother Sobolev spaces when $\gamma$ is small. This is particularly important for applications in numerical analysis and modeling, where stability is often as critical as existence and uniqueness. 

\medskip

The remainder of the paper is organized as follows. In section \ref{sec:lwp} we present the proof of Theorem \ref{thm:lwp_Hadmard} to establish local well-posedness. Section \ref{sec:nonunif} devoted to the proof of instability of the solution map which is the content of Theorem \ref{thm:nonuniformfornonperiodiccase}.  Several technical lemmas required for the main arguments are collected in Section~\ref{sec:tech_lem}.

\medskip

\section{Local wellposedness in Hadamard sense}\label{sec:lwp}
We begin with an estimate of the analogue of the Biot-Savart operator for the equations in \eqref{e1}.

\begin{Lemma}\label{Biot-Savart2} 
	Let $s>2$ and $u_{T_\gamma}=\nabla^{\perp}\Delta^{-1}T_{\gamma}\,  \theta_{T}$ be a solution of \eqref{e1} then
	\begin{equation}\label{u(Hs)-thetaH(s-1)}
		\|u_{T_\gamma}\|_{{{H}}^s}\le C\| \theta_{T}\|_{\dot{H}^{s-1}},
	\end{equation}
	for some constant $C>0$.
\end{Lemma}
\begin{proof}
By definition of the $H^s$ norm and Plancherel’s identity, we have
	\begin{align*}
		\|u_{T_\gamma}\|_{\dot{H}^s}^2&=\|\nabla^{\perp}\Delta^{-1}T_{\gamma}\,  \theta_{T}\|_{\dot{H}^s}^2\\
		&=\int\limits_{\mathbb{R}^2}|\xi|^{2s}|\mathcal{F}(\nabla^{\perp}\Delta^{-1}T_{\gamma}\,  \theta_{T})(\xi)|^2d\xi\notag\\
		&=\int\limits_{\mathbb{R}^2}|\xi|^{2(s-1)}\left|\frac{(-\xi_2,\xi_1)}{|\xi|}T_{\gamma}(|\xi|)\right|^2|\widehat{\theta}_T(\xi)|^2d\xi\notag\\
		&\le C\| \theta_{T}\|_{\dot{H}^{s-1}}^2.
	\end{align*}
	A similar calculation gives
	\begin{align*}
		\|u_{T_\gamma}\|_{L^2}^2&\lesssim \| \theta_{T}\|_{\dot{H}^{-1}}^2\lesssim \| \theta_{T}\|_{\dot{H}^{s-1}}^2.
	\end{align*} 
	Putting these estimates together,  we get the desired estimate \eqref{u(Hs)-thetaH(s-1)}.
	\end{proof}
    
To prepare for later nonlinear estimates, we recall the Littlewood–Paley decomposition and Bernstein’s inequalities.
	Let $\phi(\xi)$ be a smooth bump function supported in the ball $|\xi|\le 2$ and $\phi(\xi)=1 $ on the ball $|\xi|\le 1$. For $f\in \mathcal{S}(\mathbb{R}^2)$ and $M>0$, we define the Littlewood–Paley projection operators
	\begin{align*}
		&\widehat{P_{\le M} f}(\xi):=\phi({\xi}/{M})\hat{f}(\xi),\\
		&\widehat{P_{> M} f}(\xi):=\hat{f}(\xi)-\widehat{P_{\le M}f}(\xi),\\
		&\widehat{P_{ M} f}(\xi):=\widehat{P_{\le M}f}(\xi)-\widehat{P_{\le M/2}f}(\xi).
	\end{align*}
	
	\begin{Lemma}[Bernstein's inequalities]
		For $1\le p\le q\le \infty$, $s\in\mathbb{R}$, and for any $f\in \mathcal{S}(\mathbb{R}^2)$,  
		\begin{align*}
			&	\||\nabla|^sP_{N}f\|_{L^p(\mathbb{R}^2)}\simeq N^{s}	\|P_{N}f\|_{L^{p}(\mathbb{R}^2)},\\
			&	\|P_{\le N}f\|_{L^q(\mathbb{R}^2)}\lesssim N^{2(\frac{1}{p}-\frac{1}{q})}	\|P_{\le N}f\|_{L^{p}(\mathbb{R}^2)},\\
			&	\|P_{ N}f\|_{L^q(\mathbb{R}^2)}\lesssim N^{2(\frac{1}{p}-\frac{1}{q})}	\|P_{N}f\|_{L^{p}(\mathbb{R}^2)}.
		\end{align*}	
	\end{Lemma}
	\begin{proof}
	    See Bahouri–Chemin–Danchin \cite{BCD} or Bourgain–Li \cite{BL}.
	\end{proof}
    
    We next establish a logarithmic interpolation inequality, which will play a key role in controlling nonlinear terms.
	\begin{Lemma}\label{Log-type interpolation inequalities-2} 
			For $2<p< \infty$ and  $f\in C_{c}^{\infty}(\mathbb{R}^2),$ we have
			\begin{equation}\label{log-typeinterplation-2}
				\|D\nabla^{\perp}\Delta^{-1}f\|_{\infty}\lesssim_{p}1+\|f\|_{\infty}\log_{2}(10+\|f\|_{2}+\|\nabla f\|_p^p).
			\end{equation} 
		\end{Lemma}
		\begin{proof}	
        We split the argument into three steps.

        Step 1. Frequency decomposition.
        For integers $j<k$, set $M=2^j$, $N=2^k$.
    Using Littlewood–Paley projections, we decompose
\begin{align}
    f=P_{\le 2^j}f+ \sum\limits_{\ell=j}^{k-1} P_{2^{\ell}<\cdot\le 2^{\ell+1}}f+P_{>2^k}f
\end{align}
Applying $D\nabla^{\perp}\Delta^{-1}$ and using Bernstein’s inequalities, we obtain
\begin{align}
   \|D\nabla^{\perp}\Delta^{-1}f&\|_{\infty}\lesssim  2^j\|f\|_{2}+\sum\limits_{\ell=j}^{k-1}\| P_{2^j<\cdot\le 2^{\ell+1}}f\|_{\infty}+ \sum\limits_{\ell\ge k}2^{-\ell(1-\frac{2}{p})}\|\nabla f\|_{p}.
\end{align}
Thus, 
\begin{align}\label{e:DnDf}
   \|D\nabla^{\perp}\Delta^{-1}f&\|_{\infty}\lesssim  2^j\|f\|_{2}+(k-j)\| f\|_{\infty}+\left(2^k\right)^{-1+\frac{2}{p}}\|\nabla f\|_{p}.
\end{align}
Step 2. Choice of $j$ and $k$.

We optimize parameters so that each term in \ref{e:DnDf} is logarithmic in size
\begin{align}
    j=\log_{2}\frac{\|f\|_{\infty}}{\|f\|_{2}},\quad k = \log_{2} \frac{\left(10+\|f\|_{L^2}+\|\nabla f\|_p^p\right)\|f\|_{\infty}^{\frac{p}{2-p}}}{\|\nabla f\|_p^{\frac{p}{2-p}}}.
\end{align}
This choice ensures
\begin{equation}
\label{e:bnd_each}
\begin{aligned}
    2^j\|f\|_{2}\lesssim \|f\|_{\infty}\log_{2}\big(10+\|f\|_{2}+\|\nabla f\|_p^p\big)\\
    (k-j)\| f\|_{\infty}\lesssim  \|f\|_{\infty}\log_{2}\big(10+\|f\|_{2}+\|\nabla f\|_p^p\big)\\
    \left(2^k\right)^{-1+\frac{2}{p}}\|\nabla f\|_{p}\lesssim \|f\|_{\infty}\log_{2}\big(10+\|f\|_{2}+\|\nabla f\|_p^p\big).
\end{aligned}
\end{equation}
Substituting \eqref{e:bnd_each} into \eqref{e:DnDf}, we obtain
\begin{align}
    \|D\nabla^{\perp}\Delta^{-1}f\|_{\infty}\lesssim_{p}1+\|f\|_{\infty}\log_{2}(10+\|f\|_{2}+\|\nabla f\|_p^p)
\end{align}
which is the desired inequality.
    	\end{proof}
		\begin{Lemma}[Support decomposition under transport]\label{A1-nonperiodic}
			Let $f\in H^s\cap L^1$ with $s\ge 2$ and $g\in H^2\cap L^1$. Assume 
			\begin{align}
				\|f\|_1+\|g\|_1+\sup(\|f\|_{\infty},\|g\|_{\infty})\le M<\infty,
                \end{align}
                and suppose the supports are separated
                \begin{align}
				&\dist(\supp(f),\supp(g))\ge 100C_0M>0 \label{dist_supp_f_supp_g}
			\end{align}
			where $C_0>$ is such that
			\begin{equation*}
				\|\nabla^{\perp}\Delta^{-1}T_{}\,f\|_{L^{\infty}}\le C_0(\|f\|_{L^1}+\|f\|_{L^\infty}).
			\end{equation*}
            Let $\theta$ solve
            	\begin{equation*}
				\begin{cases}
					\partial_t \theta+(u\cdot\nabla)\theta=0,\ (x,t) \in \mathbb{R}^2\times (0,1],\\
					u=\nabla^{\perp}\Delta^{-1}T_{}\, \theta,\\
					\theta(x,0)=\theta_0.
				\end{cases}
			\end{equation*}
            Then,
            \begin{enumerate}
                \item (Decomposition) The solution can be written as $\omega = \omega_f+\omega_g$, with $\omega_f|_{t=0}=f$, $\omega_g|_{t=0}=g$.
                \item (Support propagation) For all $0\le t\le 1$,
                \begin{align}
                \supp(\omega_f(\cdot,t))\subset B(\supp(f),2C_0M),\quad\supp(\omega_g(\cdot,t))\subset B(\supp(g),2C_0M),
                \end{align}
                and
                \begin{align}
                    \dist(\supp(\omega_f(\cdot,t)),\supp(\omega_g(\cdot,t)))\ge 90C_0M.
                \end{align}
                \item (Sobolev control) For $0\le t\le 1,$\begin{equation}\label{bound_Sobolev_norm_omega_f}
					\max\limits_{0\le t\le 1}\|\omega_f(\cdot,t)\|_{H^s}\le C (\|f\|_H^{s}, M, M_1),
				\end{equation}
                \begin{equation}\label{bound_Sobolev_norm_omega_f}
					\max\limits_{0\le t\le 1}\|\omega_f(\cdot,t)\|_{H^s}\le C (\|f\|_H^{s}, M, M_1).
				\end{equation}
                where $M_1$ bounds the Lebesgue measure of 
$\supp(f)$.
            \end{enumerate}

		\end{Lemma}

      	\begin{proof}
		See Kwon \cite{HK}.
	\end{proof}

		\begin{Lemma}\label{A1-2}
			Suppose that $ \theta_{T,0}\in H^s(\mathbb{R}^2)$ for some $s\ge 2$, with $\| \theta_{T,0}\|_{H^s}\le M<\infty$ for a constant $M>1$. Let $\theta(x,t)$ be the solution to 
			\begin{equation*}
				\begin{cases}
					\partial_t \theta+(u\cdot\nabla)\theta=0,\ \mathbb{R}^2\times (0,1]\\
					u=\nabla^{\perp}\Delta^{-1}T_{}\, \theta,\\
					\theta(x,0)= \theta_{T,0}(x),.
				\end{cases}
			\end{equation*}
			Then
			\begin{equation}\label{forA2condi-2}
				\max\limits_{0\le t \le 1}\|\theta\|_{H^s}\le C
			\end{equation}
			for some constant $C=C(\| \theta_{T,0}\|_{H^s},M).$
		\end{Lemma}
		\begin{proof}
			Let $\phi(x,t)$ be  denote the flow map defined by
			\begin{equation*}
				\begin{cases}
					\partial_t \phi(x,t)=u(\phi(x,t),t)\\
					\phi(x,0)=x.
				\end{cases}
			\end{equation*}
            The transport equation then implies
			\begin{align}\label{e:transport}
            \theta(\phi(x,t),t)= \theta_{T,0}(x).
            \end{align}
            Hence the $L^p$ norms of $\theta$ are preserved it implies that $$\|\theta(\cdot,t)\|_p=~\| \theta_{T,0}\|_{p}, \quad 1\le p\le \infty,\quad  0\le t\le 1.$$ 
            
            We claim that 
            \begin{equation}\max\limits_{0\le t\le 1}\|u(\cdot,t)\|_{\infty}\le CM.\label{e:b_u_inf}
            \end{equation}
            Indeed, writing $u=K*T\theta,$ where $K$ is the Biot–Savart kernel, we split into local and nonlocal contributions. Choosing exponents $1<p<2<q<\infty$, H\"older’s inequality yields
			\begin{align*}
				\|u\|_{\infty}  
				&\le \|K\mathbbm{1}_{|\cdot|\le 1}\|_{p}\|T_{}\,\theta\|_{p'}+\|K\mathbbm{1}_{|\cdot|> 1}\|_{q}\|T_{}\,\theta\|_{q'}\\
				&\lesssim \left(\int\limits_{0}^1\dfrac{rdr}{r^p}\right)^{\frac{1}{p}}\|\theta\|_{p'}+\left(\int\limits_{1}^{\infty}\dfrac{rdr}{r^q}\right)^{\frac{1}{q}}\|\theta\|_{q'}\\
				&\lesssim C\| \theta_{T, 0}^{}\|_{p'}+C\| \theta_{T, 0}^{}\|_{q'}\\
				&\le CM.
			\end{align*}
			Observe that, from \eqref{e:b_u_inf}, we have
			\begin{equation*}
				|\phi(x,t)-x|\le \int\limits_{0}^{t}|\partial_s\phi(x,t)|ds\le \max\limits_{0\le t\le 1}\|u(\cdot,t)\|_{\infty}t\le CMt.
			\end{equation*}
            so the support of $\theta(\cdot,t)$ remains within a fixed neighborhood of the support of $\theta_{T,0}$.

			In order to estimate the Sobolev norm of $\theta$ we need to control $\nabla^{\perp}\Delta^{-1}T\theta$ when $0\le t\le 1$ and $x\in \supp(\theta(\cdot,t))$. We have 
			\begin{align*}
				\left|\partial^{\alpha}\nabla^{\perp}\Delta^{-1}T_{} \theta\right|&=\left|\int\limits_{\mathbb{R}^2}\partial^\alpha H(x-y)\theta(y)dy\right|\\
				&\le \|\partial^\alpha H\|_{L^{\infty}(\mathbb{R}^2)}\| \theta_{T,0}\|_{L^1},
			\end{align*}
			where $T=T_{\gamma}$ and $H=H_{\gamma}$ is the kernel of the Fourier multiplier $\nabla^{\perp}\Delta^{-1}T_{} $ which by Lemma \ref{kernel_for_velocity for T_gamma,beta} satisfies the estimate
			\begin{equation*}
				\left|\partial^{\alpha}H(z)\right|\lesssim_{\alpha,\gamma}\dfrac{1}{|z|^{|\alpha|+1}},\ \forall z \ne 0.
			\end{equation*}
			It follows that		\begin{equation}\label{esttransport-2}
				\max\limits_{0\le t\le 1}\max\limits_{x\in \mathbb{R}^2}	\left|\partial^{\alpha}\nabla^{\perp}\Delta^{-1}T_{} \theta\right|\lesssim_{\alpha,\gamma}1.
			\end{equation}
			Differentiating \eqref{e:transport} with respect to $x$, multiplying by $\nabla\theta\|\nabla \theta\|_{p}^{p-2}$, and integrating, we obtain the standard transport energy estimate
			\begin{equation*}
				\dfrac{1}{p}\dfrac{d}{dt}\|\nabla\theta\|_{p}^p=-\int\limits_{\mathbb{R}^2}\nabla u\nabla\theta\cdot\nabla\theta\|\nabla \theta\|_{p}^{p-2}.	
			\end{equation*}
		This implies	\begin{equation}\label{difineqLRE_p-2}
				\dfrac{1}{p}\dfrac{d}{dt}\|\nabla\theta\|_{p}^p\le \|\nabla(\nabla^{\perp}\Delta^{-1}T_{} \theta)\|_{\infty}\|\nabla\theta\|_{p}^p.
			\end{equation}
			By Lemma \ref{Log-type interpolation inequalities-2} and the $L^p$- preservation property $\|\theta(\cdot,t)\|_p=\| \theta_{T,0}^{}\|_{p}$, one has 
			\begin{align*}
				\|\nabla(\nabla^{\perp}\Delta^{-1}T_{}\theta)\|_{\infty}&\lesssim_{p} 1+\|T_{}\theta\|_{\infty}\log(10+\|T_{}\theta\|_{2}+\|\nabla T_{}\theta(\cdot,t)\|_p^p), \\
				&\lesssim_{p} 1+\|\theta\|_{\infty}\log(10+\|\theta\|_{2}+\|\nabla \theta(\cdot,t)\|_p^p),  \\
				&\lesssim_{p}1+\| \theta_{T,0}\|_{\infty}\log(10+\| \theta_{T,0}\|_{2}+\|\nabla \theta(\cdot,t)\|_p^p).
			\end{align*}
			To control  $\|\theta\|_{H^s}$ with $s\ge 2$, we apply the Kato–Ponce commutator estimate (Lemma \ref{cmrest},\cite{Dongli2019})
			\begin{align*}
				\dfrac{d}{dt}\|J^s\theta\|_{2}&\le \|[J^s, \nabla^{\perp}\Delta^{-1}T_{}\, \theta\cdot\nabla]\theta\|_{2}\\
				&\lesssim \|J^{s-1}D\nabla^{\perp}\Delta^{-1}T_{}\, \theta\|_{3}\|\nabla \theta\|_{6}+\|D\nabla^{\perp}\Delta^{-1}T_{}\, \theta\|_{\infty}\| J^{s}\theta\|_{2}\\
				&\lesssim C\|J^s\theta\|_{2},
			\end{align*}
            Using Sobolev embedding and Lemma \ref{Biot-Savart2}, we obtain
            \begin{align}
                \frac{d}{dt}\|\theta\|_{H^s}^2\le C\|\theta\|_{H^s}^2
            \end{align}
			for a constant $C$ depending on $\| \theta_{T0}\|_{H^s},\ M$, and $s$. Applying Grönwall’s inequality yields 
            \begin{align}
                \sup_{0\le t\le 1}\|\theta(\cdot,t)\|_{H^s}\le C(\| \theta_{T0}\|_{H^s},M),
            \end{align}
            which is exactly \eqref{forA2condi-2}.
		\end{proof}
  
Now, we are ready to prove Theorem \ref{thm:lwp_Hadmard}.

		\begin{proof}[Proof of Theorem \ref{thm:lwp_Hadmard}]
			We will apply Proposition \ref{thm7_Kato75} for the following equation
			\begin{equation}\label{LRE1-2}
				\partial_t \theta+(\nabla^{\perp}\Delta^{-1}T_{}\, \theta)\cdot\nabla)\theta=0,\ 	\theta(x,0)= \theta_{T,0}(x). 
			\end{equation}
			It boils down to check the conditions of Proposition \ref{thm7_Kato75}. Firstly, let's choose $Y= H^t$ for $t\ge s$, and $X=H^s$. We define the operator
\[
A(t,\theta)\theta := \big(\nabla^\perp\Delta^{-1}T_\gamma \theta(\cdot,t)\big)\cdot\nabla \theta.
\]
We also set $S=(1-\Delta)^{(t-s)/2}$. Note that 
$\operatorname{div}(\nabla^\perp \Delta^{-1}T_\gamma\theta)=0$, 
so $A(t,\theta)$ is skew-adjoint on $L^2$.
 For the condition $A_1$, we need to check
			\begin{equation*}
				\|e^{-s\{(\nabla^{\perp}\Delta^{-1}T_{}\, \theta\cdot\nabla\}}\theta\|_{H^s}^2\le e^{\beta s}\|\theta\|_{H^s},
			\end{equation*}
			for some $s\in [0,\infty)$, $t\in [0,T]$, $\theta\in W$ with $W$ will be a ball with center $0$ and the radius $R$, and $\beta$ is a real number. This condition is guaranteered by the Lemma \ref{A1-nonperiodic}. This lemma implicitly relies on the propagation of support for the flow map $ \phi(x,t) $ defined by:
\[
\partial_t \phi(x,t) = u(\phi(x,t), t), \quad \phi(x,0) = x,
\]
where $ u = \nabla^\perp \Delta^{-1} T\theta $. Crucially, for $ \theta \in W \subset H^t $ ($ t \geq s > 2 $), the velocity $ u = \nabla^\perp \Delta^{-1} T \theta $ satisfies:
\[
\|u\|_{L^\infty} + \|\nabla u\|_{L^\infty} \leq C(R), \quad R = \sup_{\theta \in W} \|\theta\|_{H^t},
\]
due to Sobolev embedding $ H^t \hookrightarrow C^1 $ for $ t > 2 $. This Lipschitz bound ensures the flow map $ \phi $ is well-defined and measure-preserving.
The operator $ A(t, \theta) $ generates the linear transport equation:
\[
\partial_t \theta + (u \cdot \nabla)\theta = 0, \quad \theta(0) = \theta_0,
\]
with solution $ \theta(t) = e^{-tA(t,\theta)} \theta_0 $. Since $ \operatorname{div} u = 0 $, the flow preserves Lebesgue measure, and:
\[
\|\theta\|_{L^p} = \|\theta_0\|_{L^p}, \quad 1 \leq p \leq \infty.
\]

For the $ H^s $ norm ($ s > 2 $), we use Lemma \ref{A1-nonperiodic}’s support control and commutator estimates (as in Lemma \ref{A1-2}’s proof). Applying $ J^s = (1 - \Delta)^{s/2} $:
\[
\frac{d}{ds} \|J^s \theta\|_{L^2}^2 = -2 \langle J^s (u \cdot \nabla \theta), J^s \theta \rangle_{L^2}.
\]

By the Kato–Ponce commutator estimate (Lemma \ref{cmrest}):
\[
\| [J^s, u] \cdot \nabla \theta \|_{L^2} \leq C \left( \|\nabla u\|_{L^\infty} \|J^s \theta\|_{L^2} + \|J^s u\|_{L^4} \|\nabla \theta\|_{L^4} \right),
\]
and since $ \operatorname{div} u = 0 $, the term $ \langle u \cdot \nabla J^s \theta, J^s \theta \rangle_{L^2} = 0 $. Using Sobolev embeddings and Lemma \ref{Biot-Savart2}:
\[
\|J^s u\|_{L^4} \lesssim \|u\|_{H^s} \leq C \|\theta\|_{H^s} = C(R), \quad \|\nabla \theta\|_{L^4} \lesssim \|\theta\|_{H^s}.
\]

Thus:
\[
\frac{d}{ds} \|J^s \theta\|_{L^2} \leq C(R) \|J^s \theta\|_{L^2},
\]
and Grönwall’s inequality gives:
\[
\|\theta(s)\|_{H^s} \leq e^{C(R)s} \|\theta_0\|_{H^s}.
\]

This implies
\[
\|e^{-s\mathcal{A}(t,\theta)}\|_{B(H^s, H^s)} \leq e^{\beta s} \quad \text{with } \beta = C(R),
\]
satisfying (A1).

			To verify $A2$, we take $S=(1-\Delta)^{\frac{t-s}{2}}$ and observe that \[A(\theta)=SA(\theta)S^{-1}-[S, A(\theta)]S^{-1},\] where
		$[S, A(\theta)]=[(1-\Delta)^{\frac{t-s}{2}}, \nabla^{\perp}\Delta^{-1}T_{} \theta\cdot\nabla].$
			For each $\theta \in H^s$, we have 
			\begin{equation*}
				\|SA(\theta)S^{-1}\theta\|_{H^s}\lesssim \|\theta\|_{H^s} 
			\end{equation*}
			\text{ by Lemma }\ref{Biot-Savart2}. Therefore, if we choose $B(\theta)=	[S, A(\theta)]S^{-1}$, then we have the following decomposition
			\begin{equation*}
				SA(\theta)S^{-1}=A(\theta)+B(\theta),
			\end{equation*}
			satisfying the condition $A2$. Indeed, since $ \nabla$ is bounded by $(1-\Delta)^{\frac{t-s}{2}}$, it follows from Lemma \ref{A2_Kato1975} that $B(\theta)$ is bounded if
			\begin{equation}\label{keyestimatesforA2condition-2}
				\nabla\cdot	\left(\nabla^{\perp}\Delta^{-1}T_{} \theta\right)\in H^{s-1}.
			\end{equation}

			This is guaranteed by \Cref{Biot-Savart2}
			and we completely verified the condition $A2$.\\	
			Next step, we check the condition $A3$. First, for each $(t,\theta)\in [0,T]\times W$, we have $A(t,\theta)\in B(H^t,H^s)$. Now we prove that the map $t\rightarrow A(t,\theta)$ is continuous in norm. Indeed, we have
			\begin{align*}
				\|\nabla^{\perp}\Delta^{-1}T_{} \theta(t)\cdot\nabla f-\nabla^{\perp}\Delta^{-1}T_{} \theta(t')\cdot\nabla f\|_{H^s}&=\|\nabla^{\perp}\left(\nabla^{\perp}\Delta^{-1}T_{} (\theta(t)-\theta(t'))\right)\cdot\nabla f\|_{H^s}\\
				&\le 	\|\nabla^{\perp}\left(\nabla^{\perp}\Delta^{-1}T_{} (\theta(t)-\theta(t'))\right)\|_{H^s}\|\nabla f\|_{H^s}\\
				&\le \|\theta(t)-\theta(t')\|_{H^s}\|f\|_{H^t},
			\end{align*}
			for all $f\in H^t$ with $t\ge s+1$. Hence, we have   
			\begin{align*}
				\|\nabla^{\perp}\Delta^{-1}T_{} \theta(t)\cdot\nabla -\nabla^{\perp}\Delta^{-1}T_{} \theta(t')\cdot\nabla\|_{H^t, H^s}\le\|\theta(t)-\theta(t')\|_{H^s}
			\end{align*}
			For each $t\in [0,T]$,  we have the map $y\rightarrow A(t,\theta)$ is Lipschitz-continuous in the sense that 
			\begin{align*}
				\|A(t,\theta_1)-A(t,\theta_2)\|_{H^t,H^s}\le \mu_1\|\theta_1-\theta_2\|_{H^s}.
			\end{align*}
			Indeed, we have
			\begin{align*}
				\|(A(t,\theta_1)-A(t,\theta_2))\theta\|_{H^s}&=\|(\nabla^{\perp}\Delta^{-1}T_{} \theta_1)\cdot\nabla\theta-(\nabla^{\perp}\Delta^{-1}T_{} \theta_2)\cdot\nabla \theta\|_{H^s}\\
				&\le \|\nabla^{\perp}\Delta^{-1}T_{} (\theta_1-\theta_2)\|_{H^s}\|\nabla \theta\|_{\infty}\\
				&\lesssim \|\theta_1-\theta_2\|_{H^s}\|\nabla\theta\|_{\infty}.
			\end{align*}

Next, we verify check condition $A4$. Let $W$ be the ball in $H^{t+1}\subset Y = H^t$ ($t\ge s>2)$ centered at $\theta_0$ of radius $R>0$, for all $ \theta \in W$ and $t \in [0, T]$,
\begin{align}
A(t, \theta)\theta_0 \in H^t \quad \text{and} \quad \|A(t, \theta)\theta_0\|_{H^t} \leq \lambda_2.
\end{align}

By Lemma \ref{Biot-Savart2}, we have
    \begin{align}\label{tag_1}
\| u \|_{H^{t+1}} \| \nabla \theta_0 \|_{H^{t-1}} \leq C_1 R \cdot R = C_1 R^2.
    \end{align}
Using the tame product estimate in $H^t$ ($t > 2$) and \eqref{tag_1}, we have
\begin{equation}\label{tag_2}
\begin{aligned}
\| u \cdot \nabla \theta_0 \|_{H^t} &\leq C \left( \| u \|_{H^{t+1}} \| \nabla \theta_0 \|_{H^{t-1}} + \| u \|_{H^t} \| \nabla \theta_0 \|_{H^t} \right)\\
&\le 2CC_1R^2=:\lambda_2. 
\end{aligned}
\end{equation}
Hence, we have 
$A(t, \theta)\theta_0 = u \cdot \nabla \theta_0 \in H^t $ and $
\| A(t,\theta)\theta_0 \|_{H^t} \leq \lambda_2$ for all $\theta \in W$ and $t \in [0, T]$.
Thus (A1)--(A4) of Proposition \ref{thm7_Kato75} are verified and we finish the proof
		\end{proof}
		\bigskip

\section{Non-uniform dependence of the data-to-solution map}\label{sec:nonunif}
We will begin with the following comparison lemma.
\begin{Lemma}\label{Consequence-Biot-Savart2} 
	Let $s>2$ and let $u_{T_\gamma}$, $u_e$ be the velocities associated with the solutions of \eqref{e1} and \eqref{e3}, respectively. We have 
	\begin{equation}\label{ue-uT(Hs)-theta_e-theta(T)H(s-1)}
		\|u_e-u_{T_\gamma}\|_{{H}^s}\le C\|\theta_e- \theta_{T}\|_{{H}^{s}}+ C\|T-\id\| \|\theta_e\|_{{H}^{s}},
	\end{equation}
	for some constant $C>0$.
\end{Lemma}
\begin{proof}
	Using the same arguments as in Lemma \ref{Biot-Savart2}, we compute
	\begin{align*}
		\|u_{T_\gamma}-u_e\|_{\dot{H}^s}^2&=\|\nabla^{\perp}\Delta^{-1}T_{}  \theta_{T}-\nabla^{\perp}\Delta^{-1} \theta_e\|_{\dot{H}^s}^2\notag\\
		&\le \|\nabla^{\perp}\Delta^{-1}T_{} ( \theta_{T}- \theta_e)\|_{\dot{H}^s}^2+ \|\nabla^{\perp}\Delta^{-1}(T_{} -\id ) \theta_e\|_{\dot{H}^s}^2\notag\\
		&\le C\| \theta_{T}- \theta_e\|_{\dot{H}^{s-1}}^2+  \|T-\id\|^2 \|\theta_e\|_{\dot{H}^{s-1}}^2 \notag \\
		&\le C\| \theta_{T}- \theta_e\|_{\dot{H}^{s}}^2+  \|T-\id\|^2 \|\theta_e\|_{\dot{H}^{s}}^2, \notag 
	\end{align*}	
	where in the last step we used $\dot{H}^{s}\hookrightarrow \dot{H}^{s-1} $.	Similarly, we find
	\begin{align*}
		\|u_{T_\gamma}-u_e\|_{L^2}^2
		&\le  \| \theta_{T}- \theta_e\|_{{H}^{s}}^2 + \|T-\id\|^2 \|\theta_e\|_{{H}^{s}}^2. \notag
	\end{align*} 
	Combining the above estimates together, we yield the desired inequality.
\end{proof}

The following lemma gives us the comparison between two vorticities corresponding to  \eqref{e1} and \eqref{e3}.
\begin{Lemma}\label{l1-2}
	Let $s> 2$ and let $T$ be as above and $\theta_{e,0}\in H^{s+1}$. We have 
	\begin{equation}\label{e5-2}
		\|\theta_e(t)- \theta_{T}(t)\|_{{H}^{s}}\le C(s,t,T) \qquad \forall\, t\ge 0
	\end{equation}
	where  
	\begin{align*}
		C(s,t,T)&= C\Big(\|\theta_{e,0}- \theta_{T,0}\|_{H^s}+ \|T-id\| \sup_{[0,t]}\|\theta_e(\tau)\|_{H^{s+1}}^2\Big) e^{t\cdot \sup\limits_{[0,t]}(\|u_{T_\gamma}(\tau)\|_{H^s}+\|\theta_e(\tau)\|_{H^{s+1}})}
	\end{align*}
\end{Lemma}

\begin{proof}
	We subtract \eqref{e1} from \eqref{e3} and then apply $J^s=(\id-\Delta)^{s/2}$ to both sides of the resulting equation. Multiplying the equation by $J^s(\theta_e- \theta_{T})$ and integrating, we obtain
	\begin{align*}\nonumber 
		\frac{1}{2}&\frac{d}{dt}\|\theta_e- \theta_{T}\|_{{H}^s}^2
		\le    \\ 
		&\le
		\big|\big\langle J^s\big((u_{T_\gamma}\cdot\nabla)(\theta_e- \theta_{T})\big), J^s(\theta_e- \theta_{T})\big\rangle_{L^2}\big|
		+
		\big|\big\langle J^s\big(((u_e-u_{T_\gamma})\cdot\nabla)\theta_e\big),  J^s(\theta_e- \theta_{T})\big\rangle_{L^2}\big|
		\\            \nonumber
		&=I_1+I_2.
	\end{align*}
	
	Using the commutator estimates \eqref{cmrest}, Sobolev's lemma and the divergence-free constraint, we have 
	\begin{align}\label{e8}
		I_1&= 	\big|\big\langle J^s\big((u_{T_\gamma}\cdot\nabla)(\theta_e- \theta_{T})\big)-\big((u_{T_\gamma}\cdot  \nabla)J^s(\theta_e- \theta_{T})\big), J^s(\theta_e- \theta_{T})	\big\rangle_{L^2}\big|\\
		&\le C \big\|J^s\big((u_{T_\gamma}\cdot\nabla)(\theta_e- \theta_{T})\big)-\big((u_{T_\gamma}\cdot \nabla)J^s(\theta_e- \theta_{T})\big)\big\|_{L^2}\|\theta_e- \theta_{T}\|_{{H}^s}\notag \\
		&\le C\big(\|\nabla u_{T_\gamma}\|_{L^\infty}\| \theta_e- \theta_{T}\|_{{H}^s} + \|u_{T_\gamma}\|_{H^s}\|\nabla (\theta_e- \theta_{T})\|_{L^\infty} \big)\|\theta_e- \theta_{T}\|_{{H}^s}\notag \\
		&\le C\big(\| u_{T_\gamma}\|_{H^s}\| \theta_e- \theta_{T}\|_{{H}^s} + \|u_{T_\gamma}\|_{H^s}\|\nabla (\theta_e- \theta_{T})\|_{H^{s-1}} \big)\|\theta_e- \theta_{T}\|_{{H}^s}\notag \\
		&\le C \| u_{T_\gamma}\|_{H^s}\| \theta_e- \theta_{T}\|_{{H}^s}^2. \notag
	\end{align}

	\bigskip
	For $I_2$, by Lemma \ref{Consequence-Biot-Savart2} we have
	
	\begin{align}\label{e9}
		I_2&\le \big\|J^s\big((u_e-u_{T_\gamma})\cdot\nabla \theta_e\big)\big\|_{L^2}\|\theta_e- \theta_{T}\|_{H^s}\\
		&\le C\big(\|u_e-u_{T_\gamma}\|_{L^\infty}\|\nabla \theta_e\|_{H^s}+\|u_e-u_{T_\gamma}\|_{H^s}\|\nabla \theta_e\|_{L^\infty}\big)\|\theta_e- \theta_{T}\|_{H^s}\notag\\
		&\le C\big(\|\theta_e- \theta_{T}\|_{H^{s}}\|\nabla \theta_e\|_{H^s}+\|\nabla \theta_e\|_{H^s}\|T-id\| \|\theta_e\|_{{H}^{s}}\big)\|\theta_e- \theta_{T}\|_{H^s}\notag\\
		&= C\big(\|\nabla \theta_e\|_{H^s}\|\theta_e- \theta_{T}\|_{H^s}^2+\|\nabla \theta_e\|_{H^s}\|T-id\| \|\theta_e\|_{{H}^{s}} \|\theta_e- \theta_{T}\|_{H^s}\big).\notag
	\end{align}
	Thus, we derive
	\begin{align}\label{e10_goal}
		\frac{1}{2}\frac{d}{dt}\|\theta_e- \theta_{T}\|_{H^s}^2&\le C \big(\| u_{T_\gamma}\|_{H^s}+ \|\nabla \theta_e\|_{H^s}\big) \|\theta_e- \theta_{T}\|_{H^s}^2\\\nonumber
		&+ C\|\nabla \theta_e\|_{H^s}\|T-id\| \|\theta_e\|_{{H}^{s}} \|\theta_e- \theta_{T}\|_{H^s}
	\end{align}
	By Lemma \ref{Lem3.3page 125: Chemin et. al.}, we get
	
	\begin{align}\label{e11}
		\|(\theta_e- \theta_{T})(t)\|_{H^s}&\le C\|\theta_{e,0}- \theta_{T,0}\|_{H^s}e^{\int\limits_{0}^{t}(\| u_{T_\gamma}(\tau)\|_{H^s}+ \|\nabla \theta_e(\tau)\|_{H^s})d\tau}\notag\\
		&+ C\int\limits_{0}^{t} \|\nabla \theta_e(\tau)\|_{H^s}\|T-id\| \|\theta_e(\tau)\|_{{H}^{s}} \big(e^{\int\limits_{\tau}^{t}(\| u_{T_\gamma}(\tau_1)\|_{H^s}+ \|\nabla \theta_e(\tau_1)\|_{H^s})d\tau_1}\big)d\tau .
	\end{align}
	Since $\theta_{e,0}\in H^{s+1}$ we have 

	\begin{align}\label{controlH_s+1ofu_e}
		\|\theta_e\|_{H^{s+1}}(t)\lesssim  \|  \theta_{0e}\|_{H^{s+1}}\exp \int\limits_{0}^{t}\| u_e(\tau)\|_{H^{s+1}}d\tau.
	\end{align}
	Using \eqref{e11} and \eqref{controlH_s+1ofu_e} we now have
	\begin{align*}
		\|(\theta_e- \theta_{T})(t)\|_{H^s}&\le  C\big(\|\theta_{e,0}- \theta_{T,0}\|_{H^s}+ \|T-id\|\cdot \sup_{[0,\tau]}\|\theta_e(\tau)\|_{H^{s+1}}^2\big)\cdot e^{t\cdot \sup_{[0,t]}\big(\|u_{T_\gamma}(\tau)\|_{H^s}+\|\theta_e(\tau)\|_{H^{s+1}}\big)}.
	\end{align*}
	Letting $C(s,t,T)$ to be the right-hand-sided above complete the proof.
	
\end{proof}
\begin{Lemma}\label{l4}
	For any $s\ge 0$ there exist two sequences of solutions $\tilde{\tilde{u}}_e^{n}(t)=\nabla ^{\perp}\Delta^{-1}\tilde{\theta}_{e}^{n}(t)$ and $\tilde{u}_e^{n}(t)=\nabla ^{\perp}\Delta^{-1}\tilde{\tilde{\theta}}_{e}^{n}(t)$  $(n=1,2,3\dots)$ of the 2D Euler equations \eqref{e3} with lifespan at least $[0,T^{\ast}]$ in $ C([0,T^{\ast}],H^s)$ such that on the one hand the initial velocity satisfy
	\begin{align}
		&\|\tilde{u}_{e,0}^{n}\|_{H^s}=\| \tilde{\tilde{u}}_{e,0}^{n}\|_{H^s}\leq 1
		\quad \text{and }\lim\limits_{n\rightarrow \infty}\|\tilde{u}_{e,0}^{n}- \tilde{\tilde{u}}_{e},0^{n}\|_{H^s}=0
	\end{align}
	
	while on the other hand for $t>0$
	\[\liminf\limits_{n\rightarrow \infty}\|\tilde{u}_e^{n}(t)- \tilde{\tilde{u}}_e^{n}(t)\|_{H^s}\ge |\sin t|>0.\]
	
\end{Lemma}
\begin{proof}
	See \cite{HM}; Thm. 2.1 and see also \cite{BL}; Thm. 1.1.
\end{proof}

\bigskip
We recall \Cref{thm:nonuniformfornonperiodiccase}:
\nonuniformfornonperiodiccase*
\begin{proof}
	Let $\tilde{u}_{e,0}^{n}$ and $ \tilde{\tilde{u}}_{e,0}^{n}$ be the initial data as in Lemma \ref{l4}. Let $\tilde{\theta}_{e}^{n}(t)=\mathrm{curl}\,\tilde{u}_{e}^{n}(t)$ and  $ \tilde{\tilde{\theta}}_{e}^{n}(t)= \mathrm{curl}\, \tilde{\tilde{u}}_{e}^{n}(t)$ be the corresponding vorticity solutions of the 2D Euler equations. 
	Using  \eqref{e1} together with Lemma \ref{Biot-Savart2} and Lemma \ref{Consequence-Biot-Savart2} we have 
	
	\begin{align}\label{perturbative est1}
		\|\tilde{u}_{T}^{n}(t)- \tilde{\tilde{u}}_{T}^{n}(t)\|_{H^s}
		&\gtrsim \|\tilde{u}_{e}^{n}(t)- \tilde{\tilde{u}}_{e}^{n}(t)\|_{{H}^s}-\underbrace{\|\tilde{\theta}_{e}^{n}(t)- \tilde{\theta}_{T}^{n}(t)\|_{{H}^{s}}}_{P_1} -\|T-\id\| \|\tilde{\theta}_{e}^{n}\|_{{H}^{s}}\notag\\
		&-\underbrace{\|\tilde{\tilde{\theta}}_{T}^{n}(t)- \tilde{\tilde{\theta}}_{e}^{n}(t)\|_{{H}^{s}}}_{P_2} -\|T-\id\| \|\tilde{\tilde{\theta}}_{e}^{n}\|_{{H}^{s}}\notag\\
		&\geq \|\tilde{u}_{e}^{n}(t)- \tilde{\tilde{u}}_{e}^{n}(t)\|_{{H}^s}-C(s,t,T) -\|T-\id\| \|\tilde{\theta}_{e}^{n}\|_{{H}^{s}} -\|T-\id\| \|\tilde{\tilde{\theta}}_{e}^{n}\|_{{H}^{s}}
		\notag\\
		&\ge |\sin t|+o(\|T-Id\|)\notag,
	\end{align}
	where $C(s,t,T):=	\max(\tilde{C},\tilde{\tilde{C}})(s,t,T)$, with $P_1\le \tilde{C}(s_T-1,t,T)$ and $P_2\le \tilde{\tilde{C}}(s_T-1,t,T)$ by the Lemma \ref{l1-2}. The last inequality follows by passing to the limit and using $ 	\liminf\limits_{n\rightarrow \infty} \|\tilde{u}_{e}^{n}(t)- \tilde{\tilde{u}}_{e}^{n}(t)\|_{H^s}\ge |\sin t|>0$ in Lemma \ref{l4} and the fact that $C(s,t,T)$ goes to $0$ by Lemma \ref{l1-2}.
\end{proof}

\section{Some technical lemmas}\label{sec:tech_lem}
We begin by recalling the well-known product estimates in the standard Sobolev space setting which we need in the periodic and non-periodic cases.
\begin{Lemma}	Let $s,t$ be real numbers such that $0<t\le s$. Then
	\begin{itemize}
		\item[(i)]\label{Appendix1_Kato1975}
		if $s>1$, then 
		$	\|fg\|_{H^t}\lesssim \|f\|_{H^s}\|g\|_{H^t}$,
		
		\item[(ii)] if $s<1$, then 
		$	\|fg\|_{H^{s+t-1}}\lesssim \|f\|_{H^s}\|g\|_{H^t}$,
		
		\item[(iii)]\label{A2_Kato1975}
		if $s>2$, then
		$	\|[\Lambda^s, M_f]\Lambda^{1-s}\|_{L^2}\lesssim\|\nabla f\|_{H^{s-1}}$.
	\end{itemize}
\end{Lemma}
\begin{proof}
	See e.g., \cite{Kato2975}; Appendix, Lem. $A1$ and $A2$.
\end{proof}
\bigskip

In sections $2$ and $3$ we use the following  commutator estimates

\begin{Lemma}(Kato-Ponce estimates.)\label{cmrest}
	Let $J^s=(id-\Delta)^{s/2}$, $s>0$, $1<p<\infty$. Then for all \/$f$, $g$ in $ \mathcal{S}(\mathbb{R}^d)$ or $C^{\infty}(\mathbb{T}^d)$, we have 
	\begin{equation}\label{KPest}
		\|J^s(fg)-fJ^s(g)\|_{L^p}\lesssim_{s,p,d}\|J^sf\|_{L^p}\|g\|_{L^{\infty}}+\|\nabla f\|_{L^{\infty}}\|J^{s-1}g\|_{L^p}.
	\end{equation}
\end{Lemma}
\begin{proof}
	See \cite{KP1988}; Appendix, Lem. $X1$.
\end{proof}
\bigskip
\begin{Lemma}(Gronwall inequality)\label{Lem3.3page 125: Chemin et. al.}
	Let $f$ and $g$ be two $C^0$ (resp., $C^1$) nonnegative functions on $[t_0, T]$. Let $A$ be a continuous function on $[t_0, T]$. Assume that, for $t\in [t_0,T]$,
	\begin{align*}
		\frac{1}{2}\frac{d}{dt}g^2(t)\le A(t)g^2(t)+f(t)g(t).
	\end{align*}
	For any time $t\in [t_0, T]$, we obtain
	\begin{align*}
		g(t)\le g(t_0)\exp\big(\int\limits_{t_0}^{t}A(s)ds\big)+\int\limits_{t_0}^{t}f(s)\exp\big(\int\limits_{s}^{t}A(\tau)d\tau\big)ds.
	\end{align*}
\end{Lemma}
\bigskip
Let $\phi(\xi)$ be a smooth bump function supported in the ball $|\xi|\le 2$ and $\phi(\xi)=1 $ on the ball $|\xi|\le 1$. Given $f\in \mathcal{S}(\mathbb{R}^d)$ and $M>0$ we have the Littlewood-Paley projection operators 
\begin{align*}
	&\widehat{P_{\le M} f}(\xi)=\phi({\xi}/{M})\hat{f}(\xi),\\
	&\widehat{P_{> M} f}(\xi)=\hat{f}(\xi)-\widehat{P_{\le M}f}(\xi),\\
	&\widehat{P_{ M} f}(\xi)=\widehat{P_{\le M}f}(\xi)-\widehat{P_{\le M/2}f}(\xi)
\end{align*}
and the following estimates
\begin{Lemma}(Bernstein's inequalities.)
	For $1\le p\le q\le \infty$ and $s\in\mathbb{R}$, we have 
	\begin{align*}
		&	\||\nabla|^sP_{N}f\|_{L^p(\mathbb{R}^d)}\simeq N^{s}	\|P_{N}f\|_{L^{p}(\mathbb{R}^d)},\\
		&	\|P_{\le N}f\|_{L^q(\mathbb{R}^d)}\lesssim_{d} N^{d(\frac{1}{p}-\frac{1}{q})}	\|P_{\le N}f\|_{L^{p}(\mathbb{R}^d)},\\
		&	\|P_{ N}f\|_{L^q(\mathbb{R}^d)}\lesssim_{d} N^{d(\frac{1}{p}-\frac{1}{q})}	\|P_{N}f\|_{L^{p}(\mathbb{R}^d)}\\
	\end{align*}
	for any $f$ in $\mathcal{S}(\mathbb{R}^d)$.
\end{Lemma}
\begin{proof}
	See e.g. \cite{BCD} or \cite{BL}.
\end{proof}
\bigskip

We next recall a general result of Kato on quasi-linear evolution equations. Let $X$ and $Y$ be reflexive Banach spaces with $Y$ continuously and densely embedded in $X$ and let $W$ be an open ball in $Y$. Let $S$ be an isomorphism from $Y$ onto $X$ and pick a norm on $Y$ so that $S$ becomes an isometry. Denote by $B(Y,X)$ the space of bounded linear operators from $Y$ to $X$.
	\begin{prop}\label{thm6_Kato75}
		Consider the following Cauchy problem in $X$
		\begin{align}\label{quasi-linear}
			&\psi_t + A(t,\psi)\psi = 0, \\
			& \psi(x,0)=\psi_0(x)\nonumber
		\end{align}
		where $\psi_0\in W$ and $0\le t\le T$. Assume that the following conditions hold.
		\begin{itemize}
			\item[($A_1$)] 
			There is $\beta \in \mathbb{R}$ such that 
			$\|e^{-sA(t,y)}\|_{B(X,X)}\le e^{\beta s}$ for any $s\in [0,\infty)$, $0\le t\le T$ and $y\in W$.
			\item[($A_2$)] For each $0\le t\le T$ and $y\in W$ we have
			\begin{align*}
				&SA(t,y)S^{-1}=A(t,y)+\tilde{A}(t,y),\\
				\text{where }	&\tilde{A}(t,y)\in B(X,X)\text{ with } \|\tilde{A}(t,y)\|_{B(X,X)}\le \lambda_1 \text{ for some constant } \lambda_1>0.\nonumber
			\end{align*}
			
			\item[($A_3$)] For each $0\le t\le T$ and $y\in W$ we have $A(t,y) \in B(Y,X)$ (in the sense that $dom(A(t,y)) \supset Y$ and the restriction of $A(t,y)$ to $Y$ is in $B(Y,X)$). For each $y\in W$ the function $t\rightarrow A(t,y)$ is continuous in the operator norm and for each $t\in [0,T]$ the function $y\rightarrow A(t,y)$ is Lipschitz continuous that is
			\begin{equation}
				\|A(t,y)-A(t,z)\|_{B(Y,X)}\le\mu_1\|y-z\|_{X},
			\end{equation}
			for some constant $\mu_1>0$ and $z\in Y$.
			\item[($A_4$)] If $y_0$ is the center of $W$ then $A(t,y)y_0\in Y$ for all  $0\le t\le T$ and $y\in W$ with
			\begin{equation}
				\|A(t,y)y_0\|_{Y}\le \lambda_2,
			\end{equation}
			for some constant $\lambda_2>0$.
			
		\end{itemize}
		Under these assumptions  \eqref{quasi-linear} has a unique solution
		\begin{equation}
			\psi\in C([0,T'];W)\cap C^1([0,T];X)
		\end{equation}
		for some $0<T'\le T$.
	\end{prop}
	
	\begin{proof}
		See \cite{Kato2975}; Thm. 6.
	\end{proof}
	\bigskip

\begin{prop}\label{thm7_Kato75}
	Under the hypotheses of Proposition \ref{thm6_Kato75},  assume that the following additional condition holds:	
	\begin{equation}
		\|\tilde{A}(t,y)-\tilde{A}(t,z)\|_{B(X,X)}\le \mu_2\|y-z\|_{Y}  	\tag{$A_5$}
	\end{equation}
	for some constant $\mu_2>0.$
		Consider the following Cauchy problems in $X$
		\begin{align}\label{quasi-linear1}
			&u^n_t + A^n(t,u^n)u^n = 0,\\ 
			&u^n(x,0)=u^n_0(x)	\nonumber
		\end{align}
		and assume that $A^n$ satisfy the above additional conditions  uniformly in $n$ (that is all constants $\beta,\lambda_1, \lambda_2,\mu_1, \mu_2$ are independent of $n$). Furthermore, assume that for each $t, y\in[0,T]\times W$, we have
		\begin{equation}
			A^n(t,y)\rightarrow A(t,y) \text{ strongly in } B(Y,X)
		\end{equation}
		\begin{equation}
			\tilde{A}^n(t,y)\rightarrow \tilde{A}(t,y) \text{ strongly in } B(X,X)
		\end{equation}
		as $n\rightarrow \infty$. If $u_0, u_0^n \in W$ and $u_0^n\rightarrow u_0$ in $Y$ as $n\rightarrow \infty$, then there exist $0<T''\le T$ and unique solutions $u^n \in C(0,T''; W)\cap C^1(0,T'';X)$  of \eqref{quasi-linear1}
		for $n=1,2,\dots$ and a unique solution $u$ of \eqref{quasi-linear} in the same class. Moreover, we have
		\begin{equation}
			u^n(t)\rightarrow u(t) \text{ in } Y,\ \text{ uniformly in } t\in[0,T''].
		\end{equation}
	\end{prop}
	\begin{proof}
		See \cite{Kato2975}; Thm. 7.
	\end{proof}
	\bigskip

	\begin{Lemma}\label{kernel_for_velocity for T_gamma,beta}
		Let $T_{\gamma}(\Lambda)=\left(\log(e-\Delta)\right)^{-\gamma}$ where $\gamma\ge 0$. Let $H_{\gamma}$ be the kernel of the opertor $\nabla^{\perp}\Delta^{-1}T_{\gamma}$.
		Then, for all $|\alpha|\ge 0$, we have
		\begin{equation}\label{est-Kernel-gamma-beta}
			|\partial^\alpha H_{\gamma}(x)|\lesssim_{\alpha} \dfrac{1}{|x|^{|\alpha|+1}}\qquad \forall x\ne 0.
		\end{equation}
		Furthermore, for any $f\in C_{c}^{\infty}(\mathbb{R}^2)$ and $1\le p\le \infty$, we have
		\begin{align}
			\|T_{\gamma}f\|_{p}\le \|f\|_{p}.\label{kernel_for_velocity for T_gamma,beta.2}
		\end{align}
	\end{Lemma}
      	\begin{proof}
		See e.g. \cite{HK}.
	\end{proof}

\end{document}